\documentclass{article}
\usepackage{graphicx}
\usepackage{subfigure}
\usepackage{amsmath}
\usepackage{amssymb}
\usepackage{amsthm}
\usepackage{float}
\usepackage[top=2cm, bottom=2cm, left=2cm, right=2cm]{geometry}

\newtheorem{theorem}{Theorem}

\usepackage{authblk}

\usepackage[compress]{cite}

\begin{document}

	\title{Chaos on the Multi-Dimensional Cube}	
	
	\author[]{Marat Akhmet\thanks{Corresponding Author, Tel.: +90 312 210 5355, Fax: +90 312 210 2972, E-mail: marat@metu.edu.tr} }
	\author[]{Ejaily Milad Alejaily}
	\affil[]{Department of Mathematics, Middle East Technical University, 06800 Ankara, Turkey}

	\date{}
	\maketitle

	\noindent \textbf{Abstract.} In this article, we show that a chaotic behavior can be found on a cube with arbitrary finite dimension. That is, the cube is a quasi-minimal set with Poincar\`{e} chaos. Moreover, the dynamics is shown to be Devaney and Li-Yorke chaotic. It can be characterized as a domain-structured chaos for an associated map. Previously, this was known only for unit section and for Devaney and Li-Yorke chaos. \\\\ 
	
	Chaos has become a very important concept that is deeply integrated into many, if not most, fields of science such as physics, biology, medicine, engineering, culture, and human activities \cite{Bolotin,Skiadas}. The chaotic behavior of some physical and biological properties was formerly attributed to random or stochastic processes or uncontrolled forces \cite{Lynch,Korsch}. Appearance of chaos in deterministic systems drew the borderline between (deterministic) chaos and stochastic noise. The idea is manifested in the chaotic behavior of simple dynamical systems. However, the randomness theory of Kolmogorov–Martin-Löfwhich still can provide a deeper understanding of the origins of deterministic chaos \cite{Bolotin}. The fundamental theoretical framework of chaos was developed in last quarter of the twentieth century. During that period, different types and definitions of chaos where formulated. In general, chaos can be defined as aperiodic long-term behavior in a deterministic system that exhibits sensitive dependence on initial conditions \cite{Kellert}. Devaney \cite{Devaney} and Li-Yorke \cite{Yorke} chaos are the most frequently used types, which are characterized by transitivity, sensitivity, frequent separation and proximality. Another common type occurs through period-doubling cascade which is a sort of route to chaos through local bifurcations \cite{Feigenbaum,Scholl,Sander}. In the papers \cite{AkhmetUnpredictable,AkhmetPoincare}, Poincaré chaos was introduced through the unpredictable point concepts. Further, it was developed to unpredictable functions and sequences. 
	
	Whoever searches in this field can discern from the literature that there is a scientific conception that chaos is everywhere. Realizing such an ideation needs to developed our mathematical tools to conceptualize all manifestation of the phenomenon. Strictly speaking, we should develop simple chaotic mechanisms that has the ability to emulate complex behaviors. Investigating the fundamental aspects of high-dimensional chaotic states is necessary in this direction. Indeed, mathematical modeling of real-world problems show that real life is very often a high–dimensional chaos and even chaotic activities in our everyday lives are difficult to described via low-dimensional systems \cite{Ivancevic}.
	
	Recently, in papers \cite{AkhmetSimilarity,AkhmetDomainStruct}, we have developed a new method of chaos formation which depends rather on the way of partition of the domain, than on a map. That is, the map is a natural consequence of the structure of the domain to be chaotic. In the present study, we extend the approach of consideration of the methodological problem on \textit{generosity} of chaos as dynamical phenomenon in real world, science and industry. This question has not been discussed in our previous researches and the present study is a complementary to those in \cite{AkhmetSimilarity,AkhmetDomainStruct}. The generosity of the phenomenon is understood in two aspects. The first one is connected to the number of models which admit chaotic dynamics. This problem is difficult to be discussed since the number of differential, discrete and other equations exhibiting chaos is still neglectingly small if you compare with the number of those admitting regular behavior. This is understandable since the development of the theory. The second aspect of the generosity concerns the density of chaotic points in the domain or the state space of the dynamics. This question has not been considered in detail (as far as we know) in literature. This is because the chaotic domains are usually considered as fractal objects, that is, there dimension is less than the dimension of the state space. Consequently it is assumed with out discussion that the points of chaos are sparse and the domain admits a small density only for one-dimensional unimodal dynamics as is the case for the unit section where it was proven that all points of the set can be chaotic \cite{Devaney,Yorke}.\\
	
	As a first step towards the goal, we show in the present work how to establish multi–dimensional chaos for simple geometrical objects. We focus in the domain structure to construct an invariant set under a chaotic map. 	Our approach is based on a recursive division of the set into an infinite number of elements which satisfy specific conditions. By using infinite sequences to index the set points, a chaotic map can be defined on the set.
	
	
	Without loss of generality, we consider $ F $ as the $ n $-dimensional unit cube, i.e., the Cartesian product of $ n $ unit intervals.
	
	For sake of comprehension, let us start with the line segment $ F = [0, 1] $. Divide $ F $ into $ 4 $ parts $ F_1=[0, \frac{1}{4}], F_2=(\frac{1}{4}, \frac{1}{2}], F_3= (\frac{1}{2}, \frac{3}{4}], F_4= (\frac{3}{4}, 1] $ (see Fig. \ref{LineConst} (b)). Divide again each part $ F_{i_1}, \; i_1=1, 2, 3, 4 $ into $ 4 $ equal parts and denote them by $ F_{i_1 i_2}, \; i_2=1, 2, 3, 4 $. Continue in this procedure such that, at the $ k^{th} $ step of the partition, each part $ F_{i_1 i_2 ... i_{k-1}} $ is divided into $ 4 $ equal parts denoted as $ F_{i_1 i_2 ... i_k}, \; i_p=1, 2, 3, 4, \; p=1, 2, ..., k $. Figure \ref{LineConst} (c) illustrates the second step of the partition of the part $ F_1 $.
	
	Considering the above simple construction, one can observe that: (i) The length of each part $ F_{i_1 i_2 ... i_k} $ approaches zero as the number of steps, $ k $, approaches infinity. This implies that an infinite iteration of the procedure would produce infinitely many points from which the line $ F $ is consisted of. The points can be represented by $ F_{i_1 i_2 ... i_k ...}, \; i_p=1, 2, 3, 4, \; p=1, 2, ... \, $. Thus, the set $ F $ can be defined as the collection of all such points, i.e.,
	\begin{equation}
	F = \big\{ F_{i_1 i_2 ... i_k ...}\; | \; i_p=1, 2, 3, 4, \; p=1, 2, ... \big\}.
	\end{equation}
	(ii) If we define the distance between two nonempty bounded sets $ A $ and $ B $ in $ F $ by $ d(A, B)= \inf \{ d(\textbf{x}, \textbf{y}) : \textbf{x}\in A, \, \textbf{y} \in B \} $, where $ d $ is the usual Euclidean metric, then
	\[ d(F_1, F_3) = d(F_1, F_4) = d(F_2, F_4) = \frac{1}{4}. \]
	
	\begin{figure}[H]
	\centering
	\subfigure[]{\includegraphics[width = 2.0in]{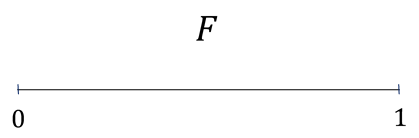}}\hspace{2cm}
	\subfigure[]{\includegraphics[width = 2.0in]{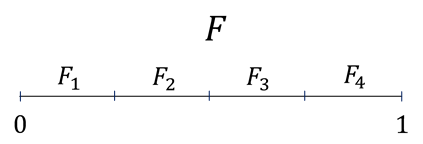}}\hspace{2cm}
	\subfigure[]{\includegraphics[width = 2.0in]{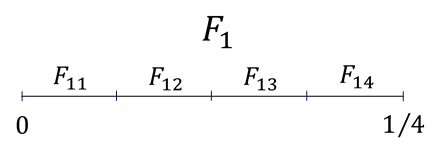}}
	\caption{The partition procedure of the line $ F $}
	\label{LineConst}   				
	\end{figure}
	
	Let us now introduce the map $ \varphi : F \to F $ defined by
	\begin{equation} \label{MapDefn}
	\varphi (F_{i_1 i_2 ... i_k ... }) = F_{i_2 i_3 ... i_k ... },
	\end{equation}
	such that for fixed sequence $ i_1 i_2 ... i_k $,  $ \varphi (F_{i_1 i_2 ... i_k}) = F_{i_2 i_3 ... i_k} $ and $ \varphi (F_{i_1}) = F $. We call each part $ F_{i_2 i_3 ... i_k} $ \textit{subset of order $ k $}, and the map $ \varphi $ the \textit{chaos generating map} or simply the \textit{generator}.
	
	Considering the results in the Appendix, one can prove that the generator is chaotic in the sense of Devaney, Li-Yorke and Poincar\`{e}. Thus, we show that a line segment can be a domain for chaos. This simple case is frankly pointed out in \cite{Devaney} for the Devaney chaos of lositic map ``$ f(x)= 4 x (1-x) $" on the interval $ [0, 1] $.
	
	As a realization of the generator map $ \varphi $, let us consider the double-humped tent map
	\begin{equation*} \label{ModTentMap}
	\varphi(x) = \left\{ \begin{array}{ll}\vspace{1mm}
	4x & \quad 0 \leq x \leq \frac{1}{4}, \\ \vspace{1mm}
	4 (\frac{1}{2}-x) & \quad \frac{1}{4} < x \leq \frac{1}{2}, \\ \vspace{1mm}
	4 (x - \frac{1}{2}) & \quad \frac{1}{2} < x \leq \frac{3}{4}, \\
	4 (1-x) & \quad \frac{3}{4} < x \leq 1.
	\end{array}	\right.\\
	\end{equation*}
	We define the first order subsets as $ F_1=[0, \frac{1}{4}], F_2=(\frac{1}{4}, \frac{1}{2}], F_3= (\frac{1}{2}, \frac{3}{4}], F_4= (\frac{3}{4}, 1] $, and for each $ i= 1, 2, 3, 4 $, we have $ F= \{ \varphi(x) : x \in F_i \} $. For each $ j= 1, 2, 3, 4 $, the second order subsets are defined by $ F_{ij}= \{ x \in F_i : \varphi(x) = F_j \} $. Generally, a $ k^{th} $ order subset can be defined by
	\[ F_{i_1 i_2 ... i_k}= \{ x \in F_{i_1 i_2 ... i_{k-1}} : \varphi(x) = F_{i_2 i_3 ... i_k} \}, \]
	where $ i_p=1, 2, 3, 4, \; p=1, 2, ...,k $. \\
	
	On the basis of the above construction, more general cases of chaotic domain can be investigated. Consider the unit square (cube) $ F $. Similarly, we divide $ F $ into $ 16 $ equal squares ($ 64 $ cubes)  and denote them by $ F_{i_1}, \; i_1=1, 2, ..., 16 $ ($ F_{i_1}, \; i_1=1, 2, ..., 64 $). Illustration of the fist step construction for a square and cube are seen in Fig. \ref{Square&Cube}. Again we divide each square (cube) $ F_{i_1} $ into $ 16 $ equal squares ($ 64 $ cubes) and denote them by $ F_{i_1 i_2}, \; i_2=1, 2, ..., 16 $ ($ F_{i_1 i_2}, \; i_2=1, 2, ..., 64 $). We continue in this procedure such that, at the $ k^{th} $ step of partition, each part $ F_{i_1 i_2 ... i_{k-1}} $ is divided into $ 16 $ equal squares ($ 64 $ cubes) denoted as $ F_{i_1 i_2 ... i_k}, \; i_p=1, 2, ..., 16, \; p=1, 2, ..., k $ ($ F_{i_1 i_2 ... i_k}, \; i_p=1, 2, ..., 64, \; p=1, 2, ..., k $). Likewise, the set $ F $ can be defined by
	\begin{equation*}
	F = \big\{ F_{i_1 i_2 ... i_k ...}\; | \; i_p=1, 2, ..., m, \; p=1, 2, ... \big\},
	\end{equation*}
	where $ m $ is $ 16 $ for the square and $ 64 $ for the cube. One should emphasize here that the number $ m $ is chosen to sufficiently satisfy the (diagonal and separation) properties mentioned in the Appendix. Analogously, in these cases, the generator $ \varphi $ is defined, and it can be verified that the assertions, in the Appendix, are applicable. Thus, the generator is chaotic in the sense of Devaney, Li-Yorke and Poincar\`{e}.
	
	\begin{figure}[H]
	\centering
	\subfigure[]{\includegraphics[width = 1.5in]{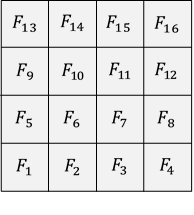}}\hspace{2cm}
	\subfigure[]{\includegraphics[width = 2.1in]{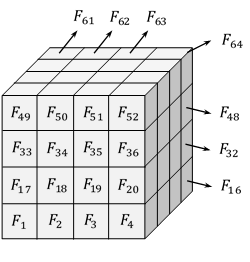}}		\caption{The partition procedure of a square and cube.}
	\label{Square&Cube}   				
	\end{figure}

	For a general case, consider the unit $ n $-dimensional cube $ F $. The first step consists of dividing $ F $ into $ 4^n $ equal parts ($ n $-dimensional sub-cubes) denoted as $ F_{i_1}, \; i_1=1, 2, ..., 4^n $. In the second step, each part $ F_{i_1} $ is again divided into $ 4^n $ equal parts denoted as $ F_{i_1 i_2}, \; i_2=1, 2, ..., 4^n $. Continue in this procedure such that, at the $ k^{th} $ step of partition, each part $ F_{i_1 i_2 ... i_{k-1}} $ is divided into $ 4^n $ equal sub-cubes denoted as $ F_{i_1 i_2 ... i_k}, \; i_p=1, 2, ..., 4^n, \; p=1, 2, ..., k $. At an infinite iteration of this process, the cube $ F $ can be represented as the collection of all points, $ F_{i_1 i_2 ... i_k ...} $, i.e.,
	\begin{equation} \label{nDCubeSet}
	F = \big\{ F_{i_1 i_2 ... i_k ...}\; | \; i_p=1, 2, ..., 4^n, \; p=1, 2, ... \big\}.
	\end{equation}
	
	Let us now verify the diagonal and separation properties stated in the Appendix. At any step $ k $, the diagonal of the sub-cube $ F_{i_1 i_2 ... i_k} $ is given by $ \frac{\sqrt{n}}{4^k} $, and since $ n $ is finite, the diagonal property holds.
	
	For the separation property, we will show that it is valid with a separation constant $ \varepsilon_0 = \frac{\sqrt{n}}{4} $. Consider an edge of the cube as projection of the cube on the dimension. The partition allows to find several sub-cubes such that the distance between them  and a sub-cube $ A $ is not less than $ \frac{1}{4} $ in the projection. Consequently, there exists a sub-cube, say $ B $, which is distanced from $ A $ in all projections not less than $ \frac{1}{4} $. This is why the distance between $ A $ and $ B $ is greater than or equal to $ \frac{\sqrt{n}}{4} $. Here, we point out that the separation constant has a direct relationship with the dimension of the cube. This simply means that the high dimensionality makes chaos more stronger, and this comports with the concept that higher state space dimensions allow for more complex attitudes including chaotic behavior \cite{Hilborn}. 
	
	Similarly, the chaos generating map is defined for the $ n $D cube and using the results in the Appendix, it can be shown that the generator defined in (\ref{MapDefn}) is chaotic in the sense of Devaney, Li-Yorke and Poincar\`{e}. 
	
	In Fig. \ref{ChaoticTrajec}, we provide simulation results illustrate the chaotic behavior of trajectories initiated at points on unit square and cube under the map $ \varphi $.
	
	\begin{figure}[H]
		\centering
		\subfigure[]{\includegraphics[width = 2.8in]{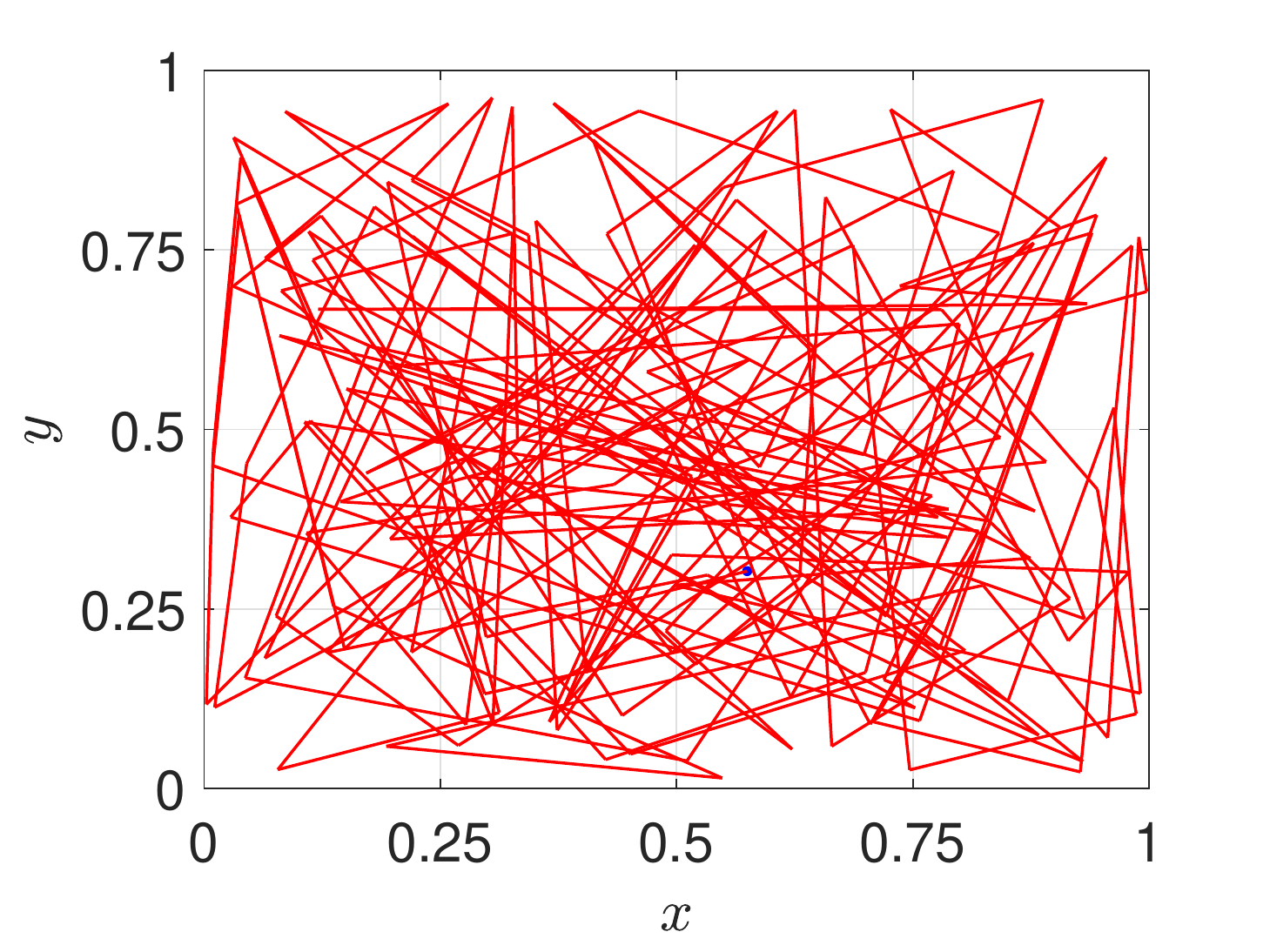}}\hspace{1cm}
		\subfigure[]{\includegraphics[width = 3.2in]{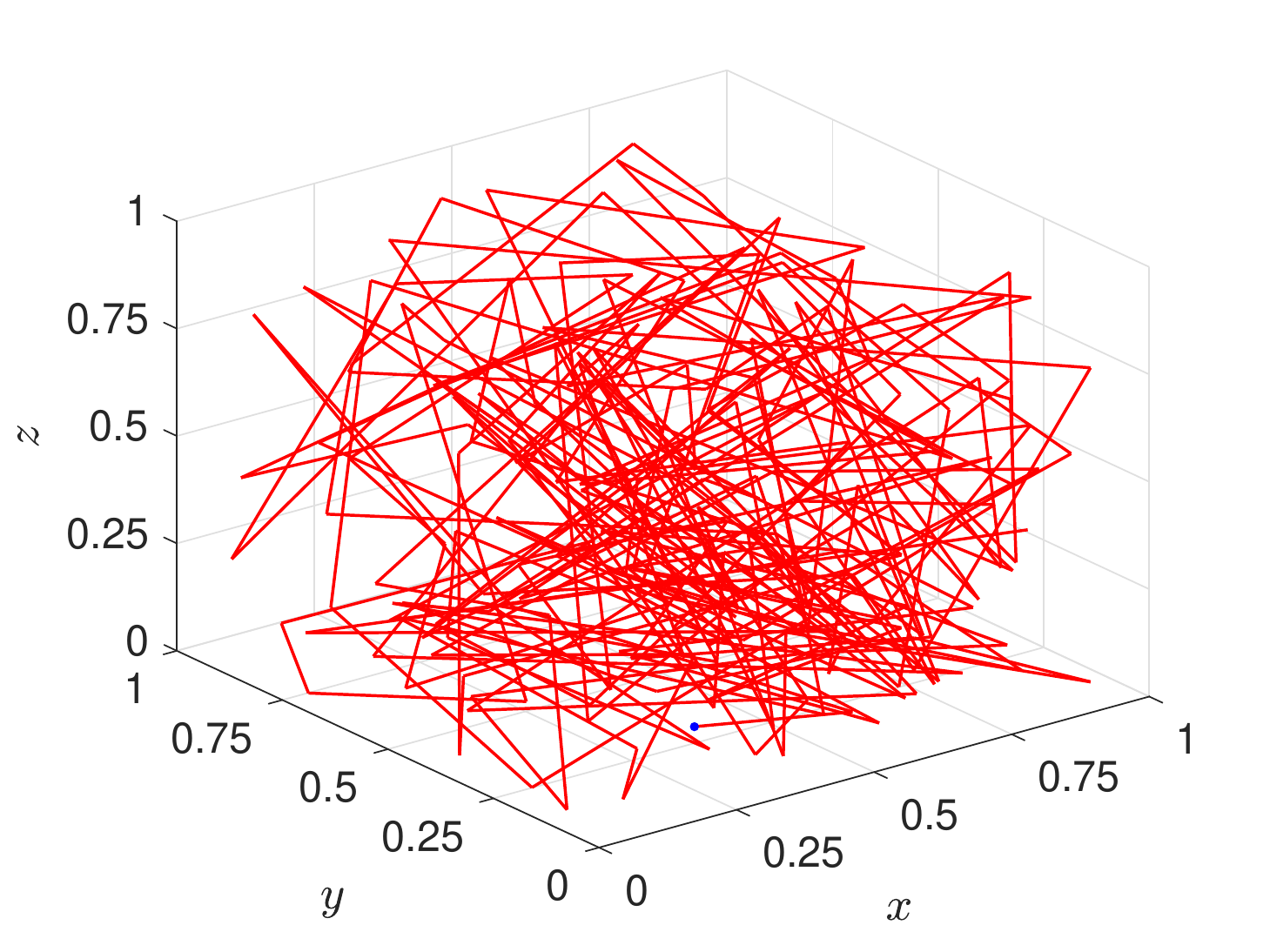}}		\caption{The chaotic behavior of trajectories: (a) initiated at the point $ \approx (0.5746337359, 0.3027738565) $ on the unit square, and (b) initiated at the point $ \approx (0.3766437951, 0.2649318230, 0.0275998135) $ on the unit cube.}
		\label{ChaoticTrajec}   				
	\end{figure}
	
	The proposed approach for generating chaotic dynamics on infinite sets sheds light on the rule of the domain of chaos. In the classical theory, the requirements and properties of chaos are usually described through the map, then, the chaotic behavior is reflected in the image which characterized by some quantitative measures such as fractional dimension, similarity and scaling. Example of such a case is the Lorenz attractor whose capacity dimension is $ 2.06 \pm 0.01 $ \cite{Lorenz}. The same could be said for an invariant subset of the domain such as a Cantor set for the logistic map which has the Hausdorff dimension $ 0.63093 $ \cite{Mandelbrot1}. In the present research we devote more attention to the domain, and less attention to properties of the map. In particular, we impose specific conditions on the domain, namely, the diagonal and separation properties. On the other hand, the continuity and injectivity requirements for the motion are ignored since a chaotic map need not be continuous or injective \cite{Keener,Osikawa,Addabbo,Degirmenci,Li}. We regard these properties of a chaotic map as important characteristics only from the analytical side, that is to say, they are very useful for handling the map to prove presence of chaos \cite{Wiggins,Chen}, however, it is not rigorously correct to consider them as intrinsic properties for chaos. Disregarding such properties helps us to organize a given infinite set as an invariant domain of the map which leads to a chaotic regime with all the domain, the map and the image have equal dimensions. This is different from the classical situation in continuous chaotic dynamics where the dimension of map is usually larger than that for the domain and image as in the case of the above two examples, Lorenz system and the logistic map.
	
	The results presented in this paper reveal that the whole points of an $ n $-dimensional cube are chaotic under the defined map, and it is clear that the approach can be extended for each infinite set which admits topological properties common to the cube. This reinforces the concept that the chaotic points are dense in the state space of dynamics of many real world phenomena. Moreover, the results would be crucial for numerous relevant analyses and applications. For example, a good application can be performed for the Brownian motion \cite{Brown,Wang} where the random movement of particles may be considered in light of the present results.

\newpage
	\section*{Appendix: Chaos for the map $ \varphi $} \label{AFChaos}
	
	Consider the Euclidean space $ \mathbb{R}^n $, and let $ F $ be the unit $ n $-dimensional cube described by (\ref{nDCubeSet}). Define the diameter of a set $ A $ in $ F $ by $ \mathrm{diam}(A) = \sup \{ d(\textbf{x}, \textbf{y}) : \textbf{x}, \textbf{y} \in A \} $. It is easy to see that the diameter of each $ n $-dimensional sub-cube $ F_{i_1 i_2 ... i_k} $ approaches zero as $ k $ approaches infinity (diagonal property), and moreover, for arbitrary $ i $ there exists $ j \neq i $ such that $ d \big( F_i \, , \, F_j \big) \geq \frac{\sqrt{n}}{4} $ (separation properties).
	
	In \cite{AkhmetUnpredictable,AkhmetPoincare} the definitions of unpredictable point and Poincar\`{e} chaos were introduced. This time, we will consider the generator map $ \varphi $ to satisfy the definitions and make conclusion that there exists an unpredictable point in $ F $ and the cube is a closure of the trajectory of the point. Moreover, we directly prove that the dynamics is sensitive. Thus, the following assertion is valid.
	
	\begin{theorem} \label{Thm1}
		The generator $ \varphi $ possesses Poincar\`{e} chaos and the cube is a quasi-minimal set for the dynamics.
	\end{theorem}

\begin{proof}
	To prove that $ \varphi $ is topologically transitive, we utilize the diagonal property to show the existence of an element $ l= F_{i_1 i_2 ... i_k ...} $ of $ F $ such that for any subset $ F_{i_1 i_2 ... i_q} $ there exists a sufficiently large integer $ p $ so that $ \varphi^p(l) \in F_{i_1 i_2 ... i_q} $. This is true since we can construct the sequence $ i_1 i_2 ... i_k ... $ such that it contains all sequences of the type $ i_1 i_2 ... i_q $ as blocks.
	
	For sensitivity, fix a point $ \mathcal{F}_{i_1 i_2 ... } \in \mathcal{F} $ and an arbitrary positive number $ \varepsilon $. Due to the diagonal property, and for a sufficiently large $ k $, the distance between $ \mathcal{F}_{i_1 i_2 ... i_k i_{k+1} ...} $ and $ \mathcal{F}_{i_1 i_2 ... i_k j_{k+1} j_{k+2} ...} $ is less than $ \varepsilon $ provided that $ i_{k+1} \neq j_{k+1} $. Since the separation property is valid, there exist an integer $ p $, larger than $ k $, such that $ \mathcal{F}_{i_{p+1} i_{p+2} ...} $ and  $ \mathcal{F}_{j_{p+1} j_{p+2} ...} $ are at a distance $ \varepsilon_0 $ apart. This proves the sensitivity.

	The proof of unpredictability is based on the verification of Lemma 3.1 in \cite{AkhmetPoincare} adopted to the chaos generating map. In a similar way for determining the element $ l $, the Lemma consists of constructing a sequence $ i_1^* i_2^* ... i_k^* ... $ to define an unpredictable point $ l^*= F_{i_1^* i_2^* ... i_k^* ...} \in F $ which satisfies Definition 2.1 in \cite{AkhmetUnpredictable}. Moreover, it can be shown that the trajectory $ \varphi(l^*) $ is dense in $ F $. This is why the cube is a quasi-minimal set and the dynamics is Poincar\`{e} chaotic.
\end{proof}

	Now, let us show that the periodic points are dense in $ F $. A point $ F_{i_1 i_2 i_3 ...} \in F $ is periodic with period $ p $ if its index consists of endless repetitions of a block of $ p $ terms. Fix a member $ F_{i_1 i_2 ... i_k ... } $ of $ F $ and a positive number $ \varepsilon $. Find a natural number $ p $ such that $ \mathrm{diam}(F_{i_1 i_2 ... i_p}) < \varepsilon $ and choose a $ p $-periodic element $ F_{i_1 i_2 ... i_p i_1 i_2 ... i_p ...} $ of $ F_{i_1 i_2 ... i_p} $. It is clear that the periodic point is an $ \varepsilon $-approximation for the considered member. Thus, the map $ \varphi $ possesses the three ingredients of Devaney chaos, namely density of periodic points, transitivity and sensitivity \cite{Devaney}.

	In addition to the Poincar\`{e} and Devaney chaos, it can be shown that the Li-Yorke chaos also takes place in the dynamics of the map $ \varphi $. The proof is similar to that of Theorem 6.35 in \cite{Chen} for the shift map defined on the space of symbolic sequences.

\end{document}